\documentclass[12pt,a4paper]{article}

\usepackage[centering, footskip=.75in]{geometry}

\usepackage{datetime}

\usepackage{tikz}

\usepackage{eucal}

\usepackage[all]{xy}

\usepackage{amssymb, amsmath, amsthm}

\usepackage[shortlabels]{enumitem}

\numberwithin{equation}{section}

\newtheorem{theorem}{Theorem}[section]
\newtheorem{proposition}[theorem]{Proposition}

\theoremstyle{definition}

\newtheorem{definition}[theorem]{Definition}
\newtheorem{remark}[theorem]{Remark}

\usepackage{todonotes}

\usepackage[pdftex]{hyperref}

\hypersetup{
  colorlinks = true,
  linkcolor = black,
  urlcolor = gray,
  citecolor = black,
}

\let\OLDthebibliography\thebibliography
\renewcommand\thebibliography[1]{
  \OLDthebibliography{#1}
  \setlength{\parskip}{0pt}
  \setlength{\itemsep}{0pt plus 0.3ex}
}

\renewcommand{\epsilon}{\varepsilon}

\newcommand{\rmC}{\mathrm{C}}

\newcommand{\rmF}{\mathrm{F}}

\newcommand{\bbZ}{\mathbb{Z}}

\newcommand{\FQ}{\mathrm{FQ}}
\newcommand{\Gr}{\mathrm{Gr}}
\newcommand{\Orb}{\mathrm{Orb}}

\title{\bf Thompson's quandle}

\author{Markus Szymik}

\newdateformat{mydate}{\monthname~\twodigit{\THEYEAR}}
\date{\mydate\today}

\begin{document}

\maketitle

%%%

We reveal that Thompson's group $F$ has a quandle refinement, and we establish some essential results about the originating quandle. 

%%%

Thompson's groups appear across mathematics in hundreds of publications devoted to them. Defying Stigler's law of eponymy, the group~$F$ first surfaced in Richard Thompson's work in logic on solvable and unsolvable problems in group theory~\cite{McKenzie+Thompson, Thompson}. Then, in homotopy theory, category theory, and shape theory, Dydak--Hastings~\cite{Dydak} and Freyd--Heller~\cite{Freyd--Heller:II}, independently, observed that the group~$F$ is universal among those with a free homotopy idempotent. Brown--Geoghegan~\cite{Brown--Geoghegan}~(see also~\cite{Zaremsky}) showed that the group~$F$ satisfies an interesting finiteness condition. Other occurrences of Thompson's groups are in surface mapping class groups and Teichm\"uller theory, see~\cite{Funar--Kapoudjian--Sergiescu, Penner}, and in connection with dynamic data storage in trees~\cite{Sleator--Tarjan--Thurston, Dehornoy:rotation}. We refer to the expositions of Cannon--Floyd--Parry~\cite{Cannon--Floyd--Parry,Cannon--Floyd} for standard facts. 

Quandles are self-distributive algebraic structures, and both Thompson's group and the self-distributivity equation have appeared in the work~\cite{Dehornoy:book,Dehornoy:Thompson} of Dehornoy. However, he connects the group~$F$ to the associative equation, and the self-distributive presentation of~$F$ seems to have eluded him. 
In a different vein, we can use quandles to classify knots, and we refer to the work of Vaughan Jones~\cite{Jones:1,Jones:2,Jones:3} for relations between knots and Thompson's group. 
Chouraqui~\cite{Chouraqui} studied relations between Thompson's group~$F$ and the Yang--Baxter equation.
Circling back to logical questions, Belk--McGrail~\cite{Belk--McGrail} discussed the word problem for quandles.

In Section~\ref{sec:def}, we first define Thompson's quandle~$P$ so that we can immediately prove its universal property in Section~\ref{sec:univ}. We use that result in Section~\ref{sec:orbits} to compute the set of orbits, which is finite, and deduce that~$P$ itself is infinite. We can then describe Thompson's quandle~$P$ as an infinite iterated~HHN extension. This requires, first, to give a definition of HNN extensions of quandles in general, and we will do both in Section~\ref{sec:HNN}. The next Section~\ref{sec:fingen} contains a proof that Thompson's quandle~$P$ is finitely presentable, followed by the result that we can recover Thompson's group~$F$ as the enveloping group~$\Gr(P)$ of Thompson's quandle~$P$ in Section~\ref{sec:F}, where we also show that the canonical map~\hbox{$P\to \Gr(P)\cong F$} is injective. In the final Section~\ref{sec:Alexander}, we compute the Alexander module of~$P$.

%%% 

\section{A definition}\label{sec:def}

Recall that a quandle is a set~$Q$ together with a binary operation~$\rhd$ such that the left-multiplications~$y\mapsto x\rhd y$ are automorphisms for all~$x$,~i.e., they are bijective and satisfy~\hbox{$x\rhd(y\rhd z)=(x\rhd y)\rhd(x\rhd z)$}, and the automorphism~$x\mapsto x\rhd x$ has to be the identity. Quandles form an algebraic theory. Therefore, the category of quandles has all limits and colimits. In particular, we can define quandles in terms of generators and relations.

We define Thompson's quandle using infinitely many generators and infinitely many relations. This ensures that we can easily derive a universal property. Further down, we will show that Thompson's quandle is finitely presented~(see Proposition~\ref{prop:fin_gen}).

\begin{definition}
The quandle~$P$ with generators~$p_0,p_1,p_2,\dots$ and relations
\begin{equation}
p_j\rhd p_k=p_{k+1}\tag{$R_{j,k}$}
\end{equation}
for all~$j<k$ is {\it Thompson's quandle}.
\end{definition}

%%%

\section{A universal property}\label{sec:univ}

For every integer~$j\geqslant 1$, we denote by~$P_j$ the subquandle generated by the~$p_k$ for~$j\leqslant k$. This way, we obtain a chain
\[
P=P_0\geqslant P_1\geqslant P_2\geqslant\dots
\]
of subquandles. Shifting indices~$p_k\leftrightarrow p_{k-j}$ gives isomorphisms~\hbox{$P_j\cong P_0=P$} for all~$j$, and we see that all subquandles~$P_j$ are isomorphic to Thompson's quandle.

The subquandle~$P_1$ of Thompson's quandle~$P=P_0$ is the image of the endomorphism~\hbox{$\epsilon\colon P\to P,\,p_n\mapsto p_{n+1}$}, and the equation
\[
\epsilon^2(p)=p_0\rhd \epsilon(p)
\]
holds for all elements~$p$ in~$P$. In other words, the endomorphism~$\epsilon$ of~$P$ is idempotent up to an inner automorphism. 

\begin{remark}
We can write~$\epsilon$ as the composition of the isomorphism~$P=P_0\cong P_1$ followed by the inclusion~$P_1\subseteq P_0=P$. It follows that~$\epsilon$ is injective.
\end{remark}

%%%

Here is a universal property of Thompson's quandle.

\begin{proposition}
If~$Q$ is a quandle together with an endomorphism~\hbox{$\delta\colon Q\to Q$} and an element~$q_0\in Q$ such that~$\delta^2(q)=q_0\rhd \delta(q)$ for all~$q$ in~$Q$, then there is a unique morphism~$f\colon P\to Q$ of quandles that maps~$p_0$ to~$q_0$ and satisfies~\hbox{$\delta f=f\epsilon$}.\end{proposition}

\begin{proof}
Let~$\delta$ and~$q_0$ be given as in the statement. We want to define
\[
f(p_j)=\delta^j(q_0)
\]
on the generators~$p_j$ of~$P$. To see that these elements of~$Q$ satisfy the required relations, we have to show that, for all~$j<k$, we have~$f(p_j)\rhd f(p_k)=f(p_{k+1})$, or~\hbox{$\delta^j(q_0)\rhd \delta^k(q_0)=\delta^{k+1}(q_0)$}. However, we already know that~$k-j\geqslant 1$, so
\[
q_0\rhd \delta^{k-j}(q_0)=\delta^{k-j+1}(q_0)
\]
by assumption~$\delta^2(q)=q_0\rhd \delta(q)$ on~$\delta$ and~$q_0$, applied to~$q=\delta^{k-j-1}(q_0)$. Applying~$\delta$ for~$j$ more times, we get
\[
\delta^j(q_0)\rhd \delta^k(q_0)=\delta^{k+1}(q_0),
\]
which is precisely what we need to define the morphism~$f\colon P\to Q$. We can then verify that it satisfies the relation~\hbox{$\delta f=f\epsilon$} by checking this equation on the generators~$p_j$ of~$P$: we have
\[
\delta f(p_j)=\delta^{j+1}(q_0)=f(p_{j+1})=f\epsilon(p_j),
\]
as required.
\end{proof}

%%%

\section{Orbits}\label{sec:orbits}

Recall that the functor that sends a quandle~$Q$ to its set~$\Orb(Q)$ of orbits is left adjoint to the functor that sends a set~$S$ to the trivial quandle on~$S$, where~\hbox{$x\rhd y=y$} for all~$x,y$ in~$S$. Specifically, we can define~$\Orb(Q)$ to be the set of equivalence classes~$[x]$ for the equivalence relation on~$Q$ generated by~$[x\rhd y]=[y]$.

\begin{proposition}\label{prop:orbits}
Thompson's quandle~$P$ has exactly two orbits, namely~$[p_0]$ and~$[p_1]=[p_2]=\dots$.
\end{proposition}

\begin{proof}
Equivalently, there is a canonical bijection between the set of morphisms~\hbox{$P\to S$} into a trivial quandle~$S$ and the set of all pairs of elements of~$S$, given by evaluating a morphism~$f$ at~$p_0$ and~$p_1$. Given~$s$ and~$t$ in~$S$, we define a morphism~$f\colon P\to S$ on generators by
\[
f(p_j)=
\begin{cases}
s & j = 0,\\
t & j\geqslant 1.
\end{cases}
\]
This assignment satisfies the requied relations, as for~$j<k$ we have~$k\geqslant 1$, and therefore
\[
f(p_j)\rhd f(p_k)=f(p_k)=t=f(p_{k+1}).
\]
This construction is inverse to the evaluation at~$p_0$ and~$p_1$.
\end{proof}

\begin{remark}\label{rem:P_infinite_1}
If we think of the quotient map~$P\to\Orb(P)$ as the universal morphism into a trivial quandle, we see that all of~$P_1$ is mapped to the element~$[p_1]$, whereas~$p_0$ is mapped to~$[p_0]\not=[p_1]$. It follows that~$p_0\in P_0\setminus P_1$, and by induction, we get~\hbox{$p_j\in P_j\setminus P_{j+1}$}. This shows that the elements~$p_j$ of~$P$ are all different, and~$P$ is infinite. Later, we shall see another argument for that in Remark~\ref{rem:P_infinite_2}.
\end{remark}

%%%

\section{HNN extensions}\label{sec:HNN}

We will now see how we can write Thompson's quandle as an infinite iterated HNN extension. HNN extensions of{\it~groups} are known well~\cite{HNN}, but as I am unaware of a definition of an HNN extension of{\it~quandles} in the literature, here is one:

\begin{definition}
If~$Q=\langle\,S\,|\,R\,\rangle$ is a quandle, and~$\tau\colon U\to V$ is an isomorphism between subquandles~$U$ and~$V$ of~$Q$, then 
\[
Q*\tau=\langle\,S\cup\{t\}\,|\,R\cup\{t\rhd u=\tau(u)\,|\,u\in U\}\,\rangle
\]
is the {\it HNN extension} of~$Q$ by~$\tau$.
\end{definition}

The HNN extension comes with a canonical morphism~$Q\to Q\star\tau$ of quandles with the property that the image of every~$p\in P\subseteq Q$ in~$Q\star\tau$ is conjugate to the image of~$\tau(p)$ in~$Q\star\tau$, even if this is not the case in~$Q$.

We can apply this now to Thompson's quandle~$P$. For the isomorphism~$\tau$ between subquandles we choose the isomorphism~\hbox{$P_1\cong P_2$} that sends~$p_j$ to~$p_{j+1}$. Then the HNN extension~\hbox{$P_1*\tau$} has a new generator we can denote by~$p_0$, and we find that~\hbox{$P_1*\tau\cong P_0$} is Thompson's quandle. Iterating the process, we get
\[
P=P_0\cong P_1\star\tau\cong P_2\star\tau\star\tau\cong P_3\star\tau\star\tau\star\tau\cong\dots,
\]
or~$P\star\tau\cong P$.

\begin{remark}
We have now seen HNN extension of quandles, and it is clear that amalgamated products~(pushouts) of quandles exist for formal reasons. Therefore, it seems plausible that there is a Bass--Serre theory of quandles that imitates the one for groups~(see~\cite{Serre}).
\end{remark}

%%%

\section{A finite presentation}\label{sec:fingen}

In this section, we show that Thompson's quandle admits a presentation with finitely many generators and finitely many relations.

\begin{proposition}\label{prop:fin_gen}
Thompson's quandle is isomorphic to the quandle defined by two generators, say~$a$ and~$b$, and two relations
\begin{gather}
a\rhd(a\rhd b)=b\rhd(a\rhd b),\label{first}\\
a\rhd(a\rhd(a\rhd b))=b\rhd(a\rhd(a\rhd b)).\label{second}
\end{gather}
\end{proposition}

\begin{proof}
Let~$Q$ be the quandle defined in the statement. Inside this quandle, we define a sequence~$q_0,q_1,q_2,\dots$ of elements inductively by setting~\hbox{$q_0=a$},~\hbox{$q_1=b$}, and
\begin{equation}\label{eq:q_relation}
q_n=q_{n-2}\rhd q_{n-1}
\end{equation}
for all~$n\geqslant 2$. We shall show below that there are quandle morphisms~\hbox{$f\colon P\to Q$} and~\hbox{$g\colon Q\to P$} with the property that~$f(p_n)=q_n$ and~$g(q_n)=p_n$. Then~$f$ and~$g$ are inverse isomorphisms.

As for~$g$, we would like to define it by~$g(a)=p_0$ and~$g(b)=p_1$.

We must check that~$p_0$ and~$p_1$ satisfy the relations~\eqref{first} and~\eqref{second}. From~($R_{0,2}$) and the definitions, we immediately get
\[
a\rhd(a\rhd b)=q_0\rhd(q_0\rhd q_1)=q_0\rhd q_2=q_3=q_1\rhd q_2=b\rhd(a\rhd b),
\] 
which is~\eqref{first}. In fact, these two relations are equivalent. Using them, we can show that also~($R_{1,3}$) and~\ref{second} are equivalent:
\begin{align*}
b\rhd(a\rhd(a\rhd b))
&=q_1\rhd q_3=q_4
=q_2\rhd q_3\\
&=(q_0\rhd q_1)\rhd(q_0\rhd q_2)\\
&=q_0\rhd(q_1\rhd q_2)\\
&=a\rhd(b\rhd(a\rhd b))\\
&=a\rhd(a\rhd(a\rhd b)).
\end{align*}
Now, we have a morphism~$g$, but so far we only know~$g(q_n)=p_n$ for~$n=0,1$. For~$n\geqslant2$, this relation follows by induction:
\[
g(q_n)=g(q_{n-2}\rhd q_{n-1})=g(q_{n-2})\rhd g(q_{n-1})=p_{n-2}\rhd p_{n-1}=p_n.
\]

To define~$f$, we have no choice but to set~$f(p_n)=q_n$, and we need to verify that these elements satisfy the relations~($R_{j,k}$) for all~$j<k$. 

We note that the definition, namely equation~\eqref{eq:q_relation}, gives~($R_{j,j+1}$) for all~$j$, and we can see as above that~($R_{0,2}$) and~($R_{1,3}$) are equivalent to~\eqref{first} and~\eqref{second}, respectively.

Now, we show that the relations~($R_{j,k-2}$) and~($R_{j,k-1}$) together imply the relation~($R_{j,k}$), whenever~$j<k-2$:
\[
q_j\rhd q_{k}
=q_j\rhd(q_{k-2}\rhd q_{k-1})
=(q_j\rhd q_{k-2})\rhd(q_j\rhd q_{k-1})
=q_{k-1}\rhd q_k=q_{k+1}.
\]
As we already know~($R_{0,1}$),~($R_{0,2}$),~($R_{1,2}$), and~($R_{1,3}$), we get~($R_{0,k}$) and~($R_{1,k}$) for all~$k$ for which these relations make sense, by induction.

Finally, we show that~($R_{0,j-1}$),~($R_{0,k-1}$),~($R_{0,k}$), and~($R_{j-1,k-1}$) imply~($R_{j,k}$) whenever~$1<j<k$:
\[
q_j\rhd q_k
=(q_0\rhd q_{j-1})\rhd(q_0\rhd q_{k-1})
=q_0\rhd(q_{j-1}\rhd q_{k-1})
=q_0\rhd q_k
=q_{k+1}.
\]
By induction, this implication gives~($R_{j,k}$) for all~$j<k$, finishing the verification that also the morphism~$f$ is well-defined.
\end{proof}

%%%

\section{The enveloping group}\label{sec:F}

If~$Q$ is any quandle, its {\it enveloping group}~$\Gr(Q)$ is defined as the group with generating set~$\{g_x\mid x\in Q\}\cong Q$ and relations~$g_{x\rhd y}=g_xg_y(g_x)^{-1}$ for all~$x,y\in Q$. This group comes with a canonical map~$Q\to\Gr(Q)$, which sends~$x$ to~$g_x$ and is both a morphism of quandles if the quandle structure on~$\Gr(Q)$ is given by conjugation, and~$\Gr(Q)$--equivariant when the group~$\Gr(Q)$ acts on the quandle~$Q$ via the quandle structure~(so that~$g_x\cdot y=x\rhd y$) and on the group~$\Gr(Q)$ itself by conjugation. The canonical map need not be injective in general. 

In this section, we explain that the enveloping group of Thompson's quandle~$P$ is Thompson's group~$F$, and we show that the canonical map~$P\to\Gr(P)\cong F$ is injective.

\begin{proposition}
The enveloping group of Thompson's quandle is Thompson's group~$F$.
\end{proposition}

\begin{proof}
The enveloping group of Thompson's quandle has the presentation
\[
\langle a, b\,|\,a\rhd(a\rhd b)=b\rhd(a\rhd b),\text{ }a\rhd(a\rhd(a\rhd b))=b\rhd(a\rhd(a\rhd b))\rangle,
\]
where~$g\rhd h=ghg^{-1}$ is conjugation. It now follows from the standard presentations of Thompson's group~$F$, such as the one given in~\cite[(1.9)]{Brown--Geoghegan}, 
that the group with this presentation is isomorphic to~$F$.
\end{proof}

\begin{remark}\label{rem:P_infinite_2}
Incidentally, the preceding proposition gives another way to argue that Thompson's quandle~$P$ is infinite~(as noted earlier in Remark~\ref{rem:P_infinite_1}). There is a commutative diagram
\[
\xymatrix{
\FQ(2)\ar[d]\ar[r] & \rmF(2)\ar[d]\\
P\ar[r] & F.
}
\]
The arrow at the top is the canonical map from the free quandle~$\FQ(2)$ on two generators to the free group~$\rmF(2)$ on two generators. This map is known to be injective, with the image being the set of conjugates of the generators. Similarly, the arrow at the bottom is the canonical map.~(We have yet to show that it is injective.) The vertical arrows are given by the two canonical generators of Thompson's quandle~$P$ and Thompson's group~$F$, respectively. These two vertical maps are surjective by construction. It follows that the image of~$P$ in~$F$ is given by the set of conjugates of the two generators in~$F$, which is infinite~(see~\cite{Belk--Matucci:1}, for instance, and the proof of Proposition~\ref{prop:injective} below).
\end{remark}

\begin{remark}
Thompson's groups~$T$ and~$V$ cannot be the enveloping groups of quandles. If~$Q$ is any non-empty quandle, then its enveloping group comes with a~(split) surjection onto an infinite cyclic group, induced by the quandle morphism~$Q\to\star$ to the terminal quandle with one element.  However, the groups~$T$ and~$V$ are simple~(in the sense of group theory) and, therefore, do not admit such a surjection.
\end{remark}

\begin{proposition}\label{prop:injective}
The canonical map~$P\to\Gr(P)\cong F$ embeds Thompson's quandle into Thompson's group.
\end{proposition}

\begin{proof}
From Proposition~\ref{prop:orbits}, we know that Thompson's quandle~$P$ has precisely two orbits, represented by~$p_0=a$ and~$p_1=b$. Identifying the enveloping group~$\Gr(P)$ of~$P$ with Thompson's group~$F$, we can use its action on~$P$ to describe~$P$ as the disjoint union~\hbox{$P\cong F/F_a\sqcup F/F_b$} of homogeneous~$F$--set, where the subgroups~$F_a$ and~$F_b$ of~$F$ are the stabilisers of~$a$ and~$b$ for the~$F$--action on~$P$. The canonical map~$P\to\Gr(P)\cong F$ is~$F$ invariant when~$F$ acts on itself via conjugation, and the image consists of the conjugates of the generators~$a$ and~$b$ inside~$F$. This is the disjoint union of the conjugacy classes of~$a$ and~$b$. We can similarly identify it with~$F/\rmC_F(a)\sqcup F/\rmC_F(b)$, where the subgroups~$\rmC_F(a)$ and~$\rmC_F(b)$ of~$F$ are the stabilisers of~$a$ and~$b$ for the~$F$--action on itself via conjugation, i.e., the centralisers of~$a$ and~$b$, respectively. The map in question is now given on orbits by the canonical maps~$F/F_a\to F/\rmC_F(a)$ and~$F/F_b\to F/\rmC_F(b)$, originating from the inclusions 
$F_a\leqslant\rmC_F(a)$ and~$F_b\leqslant\rmC_F(b)$. It now suffices to show that these inclusions are equalities, or in other words, that the centraliser of~$a$ in~$F$ stabilises~$a$ in~$P$, and similarly for~$b$. 

The centralisers in Thompson's groups have been computed first by~Guba--Sapir~\cite[Cor.~15.36]{Guba--Sapir}. We can deduce that~$\rmC_F(a)\cong\langle a\rangle$ and~$\rmC_F(b)\cong\langle b\rangle$ are infinite cyclic using~\cite[Prop.~4.1]{Belk--Matucci:2}, which says that this is the case for all elements~$f\in F$ that have no proper roots and whose reduced annular strand diagram is connected. The latter property is evident from~\cite[Fig.~5]{Belk--Matucci:2}, and the former follows from the fact that~$a$ and~$b$ represent basis elements of the abelianisation of~$F$, which is isomorphic to~$\bbZ^2$.
\end{proof}

%%%

\section{The Alexander module}\label{sec:Alexander}

For any quandle~$Q$, the abelianisation of the enveloping group is generated as an abelian group by the elements of~$Q$ modulo the relation~\hbox{$y\sim x\rhd y$}. This implies that the abelianisation is the free abelian group on the set of orbits of~$Q$. For Thompson's quandle~$P$, we have seen in Proposition~\ref{prop:orbits} that there are two orbits: the one through~$a$ and the one through~$b$, and we recover the fact that the abelianisation of Thompson's group~$F$ is isomorphic to~$\bbZ^2$. 

Instead of the abelianisation of the enveloping group, a finer invariant of a quandle is its Alexander module. This is an abelian group object in the category of quandles, or equivalently, a~$\bbZ[q^{\pm 1}]$--module. The Alexander module of a quandle~$Q$ given by a presentation~$\langle\,S\,|\,R\,\rangle$ can be computed as the~$\bbZ[q^{\pm 1}]$--module generated by~$S$ subject to the relations obtained from~$R$ by replacing~$x\rhd y$ with~$(1-q)x+qy$

\begin{proposition}
The Alexander module of Thompson's quandle~$P$, as a~$\bbZ[q^{\pm 1}]$--module, is isomorphic to~$\bbZ[q^{\pm 1}]\oplus\bbZ$.
\end{proposition}

\begin{proof}
For Thompson's quandle~$P$, we find that the Alexander module is a quotient of the free~$\bbZ[q^{\pm 1}]$--module with basis given by~$a$ and~$b$; the two relations come from~\eqref{first} and~\eqref{second}. 
The first relation~\eqref{first} becomes 
\[
(1-q)a+q((1-q)a+qb)=(1-q)b+q((1-q)a+qb),
\]
which is equivalent to
\begin{equation}\label{single}
(1-q)(a-b)=0.
\end{equation}
The second relation~\eqref{second} also has the form~$a\rhd x=b\rhd x$, which is equivalent to
\[
0=(a\rhd x)-(b\rhd x)=((1-q)a+qx)-((1-q)b+qx)=(1-q)(a-b),
\]
and is, therefore, implied by the first relation. 
The Alexander module is, therefore, generated by~$a$ and~$b$, subject to the single relation~\eqref{single}. As we can change basis from~$(a, b)$ to~$(a,a-b)$, we see that the result is
\[
\bbZ[q^{\pm 1}]\oplus(\bbZ[q^{\pm 1}]/(1-q))=\bbZ[q^{\pm 1}]\oplus\bbZ,
\]
as claimed.
\end{proof}

%%%

\section*{Acknowledgement}

I thank Nadia Mazza for inviting me to speak at the Functor Categories for Groups
meeting in Lancaster in April 2023 and Jim Belk for giving an inspiring presentation on finite germ extensions on that occasion. I thank Victoria Lebed for her feedback and valuable pointers to the literature.

%%%

\vfill

School of Mathematics and Statistics, The University of Sheffield, Sheffield S3 7RH, UNITED KINGDOM,\\
\href{mailto:m.szymik@sheffield.ac.uk}{m.szymik@sheffield.ac.uk}

Department of Mathematical Sciences, NTNU Norwegian University of Science and Technology, 7491 Trondheim, NORWAY\\
\href{mailto:markus.szymik@ntnu.no}{markus.szymik@ntnu.no}

%%%


\begin{thebibliography}{MMM88}

\bibitem[BM14]{Belk--Matucci:1} J. Belk, F. Matucci. Conjugacy and dynamics in Thompson's groups. Geom. Dedicata 169 (2014) 239--261.

\bibitem[BM23]{Belk--Matucci:2} J. Belk, F. Matucci. Conjugator length in Thompson's groups. Bull. Lond. Math. Soc. 55 (2023), no. 2, 793–810. 

\bibitem[BM15]{Belk--McGrail} J. Belk, R.W. McGrail. The word problem for finitely presented quandles is undecidable. Logic, language, information, and computation, 1--13. Lecture Notes in Comput. Sci., 9160. Springer, Heidelberg, 2015. 

\bibitem[BG84]{Brown--Geoghegan} K.S. Brown, R. Geoghegan. An infinite-dimensional torsion-free~$\mathrm{FP}_\infty$ group. Invent. Math. 77~(1984) 367--381.

\bibitem[CFP96]{Cannon--Floyd--Parry} J.W. Cannon, W.J. Floyd, W.R. Parry. Introductory notes on Richard Thompson's groups. Enseign. Math. 42~(1996) 215--256.

\bibitem[CF11]{Cannon--Floyd} J.W. Cannon, W.J. Floyd. What is... Thompson's group? Notices Amer. Math. Soc. 58~(2011) 1112--1113.

\bibitem[Cho23]{Chouraqui} F. Chouraqui. The Yang--Baxter equation and Thompson's group~$F$. Internat. J. Algebra Comput. 33 (2023) 547--584.

\bibitem[Deh00]{Dehornoy:book} P. Dehornoy. Braids and self-distributivity. Progress in Mathematics, 192. Basel, Birk\-h\"auser, 2000.

\bibitem[Deh05]{Dehornoy:Thompson} P. Dehornoy. Geometric presentations for Thompson's groups. J. Pure Appl. Algebra 203 (2005) 1--44.

\bibitem[Deh10]{Dehornoy:rotation} P. Dehornoy. On the rotation distance between binary trees. Adv. Math. 223 (2010) 1316--1355.

\bibitem[Dyd77]{Dydak} J. Dydak. A simple proof that pointed FANR-spaces are regular fundamental retracts of ANR's. Bull. Acad. Polon. Sci. S\'er. Sci. Math. Astronom. Phys. 25 (1977) 55--62. 

\bibitem[FH93]{Freyd--Heller:II} P. Freyd, A. Heller. Splitting homotopy idempotents. II. J. Pure Appl. Algebra 89 (1993) 93--106. 

\bibitem[FKS12]{Funar--Kapoudjian--Sergiescu} L. Funar, C. Kapoudjian, V. Sergiescu. Asymptotically rigid mapping class groups and Thompson groups. In: Papadopoulos, A. (eds.). Handbook of Teichm\"uller Theory, Vol. 3. European Mathematical Society, IRMA Lectures in Mathematics and Theoretical Physics, 2012.

\bibitem[GS12]{Gill-Short} N. Gill, I. Short. Conjugacy in Thompson's group~$F$. Proc. Am. Math. Soc. 141 (2012) 1529--1538.

\bibitem[GS97]{Guba--Sapir} V. Guba, M. Sapir. Diagram groups. Mem. Amer. Math. Soc. 130 (1997) 620.

\bibitem[HNN49]{HNN} G. Higman, B.H. Neumann, H. Neumann. Embedding theorems for groups. J. London Math. Soc. 24 (1949) 247--254. 

\bibitem[MT73]{McKenzie+Thompson} R. McKenzie, R.J. Thompson. An elementary construction of unsolvable word problems in group theory. Word Probl., Decision Probl., Burnside Probl., 457--478. Group Theory, Studies Logic Foundations Math., 71. 1973.

\bibitem[Jon16]{Jones:1} V.F.R. Jones. Knots, groups, subfactors and physics. Jpn. J. Math. 11 (2016) 69--111. 

\bibitem[Jon17]{Jones:2} V.F.R. Jones. Some unitary representations of Thompson's groups~$F$ and~$T$. J. Comb. Algebra 1 (2017) 1--44. 

\bibitem[Jon19]{Jones:3} V.F.R. Jones. On the construction of knots and links from Thompson's groups. Knots, low-dimensional topology and applications, 43--66. Springer Proc. Math. Stat. 284. Springer, Cham, 2019.

\bibitem[Pen12]{Penner} R.C. Penner. Decorated Teichm\"uller theory. The QGM Master Class Series. Z\"urich: European Mathematical Society (EMS), 2012.

\bibitem[Ser77]{Serre} J.-P. Serre. Arbres, amalgames,~$\mathrm{SL}_2$. R\'edig\'e avec la collaboration de H. Bass. Ast\'erisque, No. 46. Soci\'et\'e Math\'ematique de France, Paris, 1977.

\bibitem[STT88]{Sleator--Tarjan--Thurston} D.D. Sleator, R.E. Tarjan, W.P. Thurston. Rotation distance, triangulations, and hyperbolic geometry. J. Am. Math. Soc. 1 (1988) 647--681.

\bibitem[Tho80]{Thompson} R.J. Thompson. Embeddings into finitely generated simple groups which preserve the word problem. Word problems II, 401--441. Stud. Logic Found. Math., 95. 1980.

\bibitem[Zar21]{Zaremsky} M.C.B. Zaremsky. A short account of why Thompson's group~$F$ is of type~$\mathrm{F}_\infty$. Topology Proc. 57 (2021) 77--86.

\end{thebibliography}
\end{document}